\documentclass[a4paper,11pt]{amsart}
\usepackage{cases}
\usepackage{bm}
\usepackage{amsfonts,amsmath,amssymb,amscd,bbm,amsthm,mathrsfs,dsfont}
\usepackage{hyperref}
\hypersetup{
	colorlinks=true,
	linkcolor=blue,
	filecolor=magenta,
	urlcolor=cyan,
	citecolor=red
}
\usepackage{cleveref}
\usepackage[all]{xy}
\SelectTips{cm}{}

\usepackage{graphicx}
\usepackage{psfrag}

\usepackage{tikz}
\usepackage{tikz-cd}
\usepackage{arydshln}

\usepackage{amssymb}
\usetikzlibrary{decorations.pathreplacing}
\usetikzlibrary{matrix,arrows}
\usepackage{enumitem}

\allowdisplaybreaks[3]

\newcommand{\ie}{{\em i.e.}\ }

\newtheorem{theorem}{Theorem}[section]
\newtheorem{lemma}[theorem]{Lemma}
\newtheorem{definition}[theorem]{Definition}
\newtheorem{proposition}[theorem]{Proposition}
\newtheorem{corollary}[theorem]{Corollary}

\newtheorem{example}[theorem]{Example}
\newtheorem{remark}[theorem]{Remark}

\newcommand{\reminder}[1]{}

\newcommand{\opname}[1]{\operatorname{\mathsf{#1}}}

\newcommand{\supp}{\opname{supp}}

\newcommand{\sgn}{\mbox{sgn}}

\renewcommand{\tilde}[1]{\widetilde{#1}}

\makeatletter
\@addtoreset{equation}{section}
\makeatother

\begin{document}

\title{Conjugation on reddening sequences and conjugation difference}
\author{Siyang Liu}
\address{School of Mathematics, Hangzhou Normal University, Hangzhou, China}
\email{siyangliu@hznu.edu.cn}

\author{Jie Pan}
\address{Department of Mathematics, Faculty of Sciences, University of Sherbrooke, Sherbrooke, Quebec, Canada}
\email{jie.pan@usherbrooke.ca}

\begin{abstract}  
We describe the conjugation of the reddening sequence according to the formula of $c$-vectors with respect to changing the initial seed. As applications, we extend the Rotation Lemma, the Target before Source Theorem, and the mutation invariant property of the existence of reddening sequences to totally sign-skew-symmetric cluster algebras. Furthermore, this also leads to the construction of conjugation difference which characterizes the number of red mutations a maximal green sequence should admit in any matrix pattern with the initial seed changed via mutations.

\end{abstract}



\maketitle

%\tableofcontents

\section{Introduction}	
Fomin and Zelevinsky introduced totally sign-skew-symmetric cluster algebras in \cite{fomin2002fundations}, which include skew-symmetric cluster algebras and skew-symmetrizable cluster algebras as significant examples. In the past more than twenty years, a great deal of results were presented in numerous papers, and most of them mainly focused on skew-symmetric cases or skew-symmetrizable cases due to their connection with triangulated surfaces, categories, scattering diagrams, etc.. However, although Fomin and Zelevinsky already found many non-skew-symmetrizable totally sign-skew-symmetric cluster algebras in \cite{fomin2002fundations}, people do not develop sufficiently effective methods to deal with them until Huang and Li unfolded a large collection of totally sign-skew-symmetric cluster algebras, which contains all acyclic cluster algebras, as skew-symmetric cluster algebras of infinite rank to prove that the former inherit a majority of properties of the latter via a projection induced by the unfolding in \cite{huang2018acyclic}. It was also in this paper that they confirmed every acyclic sign-skew-symmetric matrix is totally mutable, which was conjectured by Berenstein, Fomin, and Zelevinsky in \cite{berenstein2005upperbounds}.  It then provided a large class of non-skew-symmetrizable totally sign-skew-symmetric matrices. Later, Li, Liu, and Pan further discussed totally sign-skew-symmetric cluster algebras and found many important properties can be extended to this general case (cf., \cite{li2022polytope, liu2024acyclic, pan2023mutation}). These suggest that we could return to the Laurent expression of cluster variables and mutations of seeds to talk about cluster structures more algebraically or more combinatorially to overcome the limitation of any connections mentioned above, which is also one principle of this paper. 

The maximal green sequence, introduced by Keller in \cite{keller2011mgs} to compute combinatorial  Donaldson-Thomas invariants of quivers and quantum dilogarithms, is a particular sequence of mutations given by a sequence of indices of the matrix based on the sign-coherence of $c$-vectors. Since the sign-coherence of $c$-vectors was generalized to totally sign-skew-symmetric cluster algebras in \cite{li2022polytope}, this notation can be defined naturally for totally sign-skew-symmetric matrices. 

Although the definition seems simple, maximal green sequences are rather mysterious. Even the existence of maximal green sequences is still an open problem. Partial answers were obtained in different contexts. For example, Br\"ustle, Dupont and P$\mathrm{\acute{e}}$rotin showed quivers of finite type or acyclic quivers admit a maximal green sequence in \cite{brustle2014mgs}; Mills constructed a maximal green sequence for each quiver arising from a triangulation of some marked surface which is not once-punctured and closed in the sense of \cite{fomin2008surfaces} which claimed the existence of maximal green sequence for quivers of mutation-finite type unless it arises from a triangulation of a once-punctured closed surface, or is in the mutation class of $\mathbb X_7$ in \cite{mills2017mgs}. 

Hence it was conjectured the existence of maximal green sequence is closely related to some properties of cluster algebras, such as the equality of a cluster algebra with its corresponding upper cluster algebra. Since a cluster algebra corresponds to a mutation equivalent class of exchange matrices, it is natural to study the existence of maximal green sequence under matrix mutations. Muller's results in \cite{muller} showed that the existence of maximal green sequences is not mutation invariant. However, he also proved that the existence of a weaker version of maximal green sequences, namely reddening sequences, is mutation invariant using scattering diagrams via conjugation. Then, here comes a natural question: how is a mutation sequence changed under conjugation? In this paper, we analyze their conjugation behavior via mutation of $c$-vectors for a better understanding of maximal green sequences or reddening sequences.

This has already been partially studied for skew-symmetrizable cluster algebras. In \cite{brustle2017semipicture}, Br\"ustle, Hermes, Igusa, and Todorov proved if a skew-symmetrizable matrix admits a maximal green sequence, then so does every matrix under matrix mutations along this sequence, which is called the {\em Rotation Lemma} (see Corollary \ref{rotationlemma}), and they used it to prove a conjecture on maximal green sequences called the {\em  Target before Source Conjecture} (see Theorem \ref{tbsc}). They also proved another similar theorem for any skew-symmetrizable matrix, see Theorem \ref{tbscc} for details, which we call the {\em Target before Source Theorem of $c$-vectors}. An analogue of Rotation Lemma for Jacobian algebras of Dynkin quivers was formulated by Igusa in \cite{igusa2019rotation} and for any finite dimensional Jacobian algebra by Li and Liu in \cite{liu2020rotation}. 

In this paper, we calculate how a $c$-vector changes under mutations of the initial seed and find that it is analogous to the simple reflection acting on simple roots -- the mutation in direction $k$ only changes the sign of $c$-vectors equal to $\pm e_k$. This indicates the \emph{Conjugation Lemma} (see Theorem \ref{conj}) claiming how the color of a mutation is changed under conjugation, which then helps us to extend the Rotation Lemma to totally sign-skew-symmetric matrices as rotation can be regarded as the composition of conjugation and cancellation of two consecutive mutations in the same direction and leads to the result of Muller in \cite{muller} for totally sign-skew-symmetric cluster algebras: the conjugation of a reddening sequence is a reddening sequence of the mutated initial seed, hence the existence of reddening sequences is mutation invariant.

Furthermore, we construct the conjugation difference of each vertex in a rooted $n$-regular tree, which precisely measures the change of the number of red mutations admitted by a reddening sequence under changing the initial seed and thus shows what a maximal green sequence of any matrix in the mutation equivalent class should be like in a given matrix pattern with respect to the conjugation difference.

As applications, we then prove the Target before Source Theorem for acyclic sign-skew-symmetric matrices and the Target before Source Theorem of $c$-vectors for any totally sign-skew-symmetric matrix using the Rotation Lemma and the hereditary property of the existence of maximal green sequences. Our proof is quite combinatorial and does not depend on the categorification of cluster algebras.  

The paper is organized as follows. In Section \ref{section: preliminary}, we give basic definitions of totally sign-skew-symmetric matrices and maximal green sequences as well as some properties on $C$-matrices and $G$-matrices. In Section \ref{section: conjugation}, we prove the Conjugation Lemma and the Rotation Lemma for totally sign-skew-symmetric matrices, which depict the conjugation of reddening sequences. In Section \ref{section: potential}, we construct the conjugation difference which measures the theoretical minimum of red mutations a reddening sequence admits. We finally extend the Target before Source Theorem to totally sign-skew-symmetric algebras in Section \ref{section: target before source}.

\section{Preliminaries}\label{section: preliminary}
 In this paper, all matrices are integer matrices and for a natural number $n$, we denote the set $\{1,2,\dots,n\}$ by $[1,n]$, and for a real number $a$, we denote $[a]_+:=\mathrm{max}\{a,0\}$.  Let $I$ be the $n\times n$ identity matrix, and we often denote the $j$-th column of $I$ by $e_j$ for $j\in [1,n]$.
\subsection{Matrix patterns}
 For any $n\times n$ matrix $B=(b_{ij})_{n\times n}$, we associate a directed (simple) graph $\Gamma_B$ whose vertices are given by $1,2,\dots,n$ and there is a directed edge $i\rightarrow j$ if and only if $b_{ij}>0$. 
\begin{definition}
Let $B=(b_{ij})_{n\times n}$ be an $n\times n$ matrix. We say
	\begin{enumerate}
		\item[-]  $B$ is {\em skew-symmetrizable} if there is a diagonal matrix $D=\text{diag}(d_i, i\in [1,n])$ with positive integer diagonal entries such that $d_ib_{ij} = -d_jb_{ji}$ for any $i, j \in [1,n]$, and $D$ is called a {\em skew-symmetrizer of $B$}.
		\item[-]  $B$  is  {\em sign-skew-symmetric} if either $b_{ij}b_{ji} <0$ or $b_{ij} = b_{ji} = 0$ for any $i, j \in [1,n]$.
		\item[-]  $B$ is  {\em non-positive}, if $b_{ij}\leq 0$ for all $i, j \in [1,n]$; $B$ is  {\em non-negative}, if $b_{ij}\leq 0$ for all $i, j \in [1,n]$, and
		\item[-]  $B$ is  {\em acyclic}, if $\Gamma_B$ is acyclic, \ie, there are no oriented cycles in $\Gamma_B$.
			\end{enumerate}                                                                                         
\end{definition}

 Skew-symmetric matrices are skew-symmetrizable, and skew-symmetrizable matrices are sign-skew-symmetric. 

\begin{definition}
    Let $B=(b_{ij})_{n\times n}$ be a sign-skew-symmetric matrix, and let $j\in [1,n]$. We say the index $j$ is a {\em source}  of $B$ if $b_{ji}\geq 0$ for any $i\in [1,n]$, and we say the index $j$ is a {\em sink} of $B$ if $b_{ji}\leq 0$ for any $i\in [1,n]$.
\end{definition}

   Let $B=(b_{ij})_{n\times n}$ be an arbitrary matrix. For $k\in [1,n]$, define another matrix $B' =(b'_{ij})_{n\times n}:= \mu_k(B)$  which is called the {\em mutation of $B$ at $k$} and obtained from $B$ by the following rules: 
\begin{equation}
b'_{ij}=\begin{cases} -b_{ij},&\text{if $i=k$ or $j=k$,}\\b_{ij} + \frac{|b_{ik}|b_{kj}+b_{ik}|b_{kj}|}{2},&\text{otherwise.}\end{cases}	 
\end{equation}
   It is easy to check that $\mu_k\mu_k(B) = B$. For a sign-skew-symmetric matrix $B$,  $k$ is a source of $B$ if and only if $k$ is a sink of $\mu_k(B)$, and in this case, $\mu_k(B)$ is obtained from $B$ by replacing $b_{ik}$ with $-b_{ik}$ and  $b_{kj}$ with $-b_{kj}$ for any $i,j$.

   Matrix mutations give rise to an equivalence relation among square matrices which are called {\em mutation equivalence} and usually denoted by $B'\sim B$. The mutation of a sign-skew-symmetric matrix is not necessarily sign-skew-symmetric, so Berenstein, Fomin, and Zelevinsky introduced the following concept in \cite{berenstein2005upperbounds}.

\begin{definition}
	A matrix $B$ is called {\em totally mutable} if every matrix mutation equivalent to $B$ is sign-skew-symmetric. In this case, we also say $B$ is a {\em totally sign-skew-symmetric matrix}.
 \end{definition}
 
 If $B$ is skew-symmetrizable, then $\mu_k(B)$ is skew-symmetrizable with the same skew-symmetrizer as $B$ for any $k\in[1,n]$. So skew-symmetrizable matrices are totally sign-skew-symmetric, and there are lots of totally sign-skew-symmetric matrices which are not skew-symmetrizable. Indeed, Huang and Li proved that acyclic sign-skew-symmetric matrices are totally sign-skew-symmetric in \cite{huang2018acyclic}, and Fomin and Zelevinsky gave many examples of non-skew-symmetrizable totally sign-skew-symmetric matrices in [Proposition 4.7, \cite{fomin2002fundations}]. The matrices given in [Proposition 4.7, \cite{fomin2002fundations}] are mutation-cyclic in the following sense.
 
  \begin{definition} Let $B$ be a totally sign-skew-symmetric matrix. We say $B$ is {\em mutation-acyclic} if $B$ is mutation equivalent to an acyclic matrix. Otherwise, we say $B$ is {\em mutation-cyclic}.
 \end{definition} 
 
 \begin{definition}
 	Let $B=(b_{ij})_{n\times n}$ be an acyclic sign-skew-symmetric matrix. A sequence of indices $(j_1,j_2,\dots,j_n)$ is called an {\em admissible sequence of sources of $B$}, if the following conditions are satisfied:
 	\begin{enumerate}
 		\item[(i)] $\{j_1,j_2,\dots,j_n\}$ is a permutation of $\{1,2,\dots,n\}$;
 		\item[(ii)] $j_1$ is a source of $B$, and $j_k$ is a source of $\mu_{j_{k-1}}\dots \mu_{j_2}\mu_{j_1}(B)$ for $2\leq k\leq n$. 
 	\end{enumerate}
 \end{definition}
 
The above definition is equivalent to say $l<k$ if $b_{j_lj_k}>0$ for $l,k\in [1,n]$. Notice that every acyclic sign-skew-symmetric matrix admits an admissible sequence of sources.
 
 Let $\mathbb{T}_n$ be the $n$-regular tree whose edges emanating from each common vertex are labeled by $1,2,\dots, n$.  
 
 \begin{definition}
 	A  {\em matrix pattern (\textit{or a} tropical cluster pattern) on $\mathbb{T}_n$ at the initial vertex $t_0\in \mathbb T_n$} is an assignment  of a triple $\Sigma_t=(B_t, C_t, G_t)$, which is called a \emph{seed}, to every vertex $t\in \mathbb{T}_n$ satisfying the following conditions:
 \begin{enumerate}
  \item For each vertex $t\in \mathbb{T}_n$, the triple $\Sigma_t=(B_t, C_t, G_t)$ consists of three $n\times n$ integer matrices. We call the matrix $B_t=(b_{ij,t})_{n\times n}$ the {\em exchange matrix}, we call the matrix 
  $C_t=(c_{ij,t})_{n\times n}$ a {\em $C$-matrix} and its column vectors are called {\em  $c$-vectors}, and we call the matrix 
  $G_t=(g_{ij,t})_{n\times n}$ a {\em $G$-matrix} and its column vectors are called {\em $g$-vectors}.
\item On the initial vertex $t_0$, the exchange matrix $B_{t_0}$  is a totally sign-skew-symmetric matrix, and   $C_{t_0}$ and $G_{t_0}$ are the identity matrix $I$. 
\item For each edge $t\frac{k}{\quad \;}t'$ on $\mathbb T_n$,   the triple $\Sigma_{t'}=(B_{t'}, C_{t'}, G_{t'})$ is obtained from $\Sigma_t=(B_t, C_t, G_t)$ by the following rules:
\begin{enumerate}
	\item[-]  $B_{t'}=\mu_k(B_t)$.

	\item[-] $C_{t'}=(c_{ij,t'})_{n\times n}$ is  obtained from $\Sigma_t$ by the following rule:

\begin{equation}\label{ccmut}
c_{ij, t'}=\begin{cases} -c_{ik, t},&\text{if $j=k$;}\\c_{ij, t} + \frac{c_{ik,t}|b_{jk,t}|+ |c_{ik,t}|b_{jk,t}}{2}  ,&\text{otherwise.}\end{cases}	 
\end{equation}

	\item[-]   $G_{t'}=(g_{ij,t'})_{n\times n}$ is  obtained from $\Sigma_t$ by the following rule:
\begin{equation}\label{gmut}
 g_{ij, t'}=\begin{cases} g_{ij,t},&\text{if $j\neq k$;}\\-g_{ik,t} + \sum\limits_{s=1}^ng_{is,t}[-b_{sk,t}]_+- \sum\limits_{s=1}^nb_{is,t}[-c_{sk,t}]_+,&\text{otherwise.}\end{cases}
 \end{equation}
\end{enumerate}
\end{enumerate}
\end{definition}

An immediate check shows that $\mu_k\mu_k(\Sigma_t) = \Sigma_t$ and a matrix pattern is determined by the initial totally sign-skew-symmetric matrix $B_{t_0}$.  
 We will use the notation $(B_t,C_t^{t_0},G_t^{t_0})_{t\in \mathbb T_n}$ to denote a matrix pattern on $\mathbb T_n$ at the initial vertex $t_0\in \mathbb T_n$, and we also frequently denote a matrix pattern on $\mathbb T_n$ by $(B_t,C_t,G_t)_{t\in \mathbb T_n}$ for simplicity.
For any $k\in [1,n]$ and $t\in \mathbb T_n$, we denote $c^{t_0}_{k,t}$ the $k$-th column of the $C$-matrix $C_t^{t_0}$ and $g^{t_0}_{k,t}$ the $k$-th column of the $G$-matrix $G_t^{t_0}$.

The \emph{exchange graph} associated to a matrix pattern is the quotient graph of $\mathbb T_n$ by relations that $t\sim t'$ if $(B_t,C_t,G_t)=(B_{t'},C_{t'},G_{t'})$ up to a permutation.

\subsection{Properties of $C$-matrices and $G$-matrices}
A nonzero vector is said to be {\em sign-coherent}  if it has either all non-negative coordinates or all non-positive coordinates.  For any matrix pattern, Fomin and Zelevinsky proved the first duality in \cite{fomin2007coefficients} and they conjectured every $c$-vector is sign-coherent and every row vector of the $G$-matrix is sign-coherent in \cite{fomin2007coefficients}. The sign-coherence conjecture was solved for skew-symmetrizable matrices in \cite{gross2018bases, nakanishi2012duality} and for totally sign-skew-symmetric matrices in \cite{li2022polytope}. The second duality for skew-symmetrizable matrices was proved in \cite{nakanishi2012duality}. 
 Here we summarize some important properties of matrix pattern as follows. 

\begin{lemma}
	Let $(B_t,C_t,G_t)_{t\in \mathbb T_n}$ be a matrix pattern at the initial vertex $t_0$. Then 
	\begin{enumerate}
	\item[(i)] (First duality) for any $t\in \mathbb T_n$, we have that  \begin{equation}\label{firstduality}
		G_tB_t = B_{t_0}C_t.
	\end{equation}
	\item[(ii)] (Sign-coherence)  for any $t\in \mathbb T_n$, the column vectors of $C_t$ and the row vectors of $G_t$ are sign-coherent.
	\end{enumerate}
		If in addition, $B_{t_0}$ is a skew-symmetrizable matrix with the diagonal matrix $D$, then 
	\begin{enumerate}
	\item[(iii)] (Second duality for skew-symmetrizable matrices)
	for any $t\in \mathbb T_n$, we have that  \begin{equation}
		D^{-1}G_tDC_t = I.
		\end{equation}
	\end{enumerate}
\end{lemma}

 Due to the sign-coherent property of $C$-matrices and $G$-matrices, we will use $\varepsilon_k(C_t)$ to denote the sign of the $k$-th column of the $C$-matrix $C_t$ and  $\varepsilon_k(G_t)$ to denote the sign of the $k$-th row of the $G$-matrix $G_t$ for any $k\in [1,n]$ and any vertex $t\in \mathbb T_n$. Then we can orient the exchange graph with respect to the initial vertex $t_0$ such that each arrow corresponds to a green mutation and denote it as $E_{t_0}$. For each arrow $\alpha$ of $E_{t_0}$, denote by $\alpha^{-1}$ the opposite arrow of $\alpha$. Slightly abusing of notion, we call $\alpha_{i_l}^{\epsilon_l}\cdots\alpha_{i_1}^{\epsilon_1}$ a path in $E_{t_0}$ if the target of $\alpha_{i_j}^{\epsilon_j}$ coincides with the source of $\alpha_{i_{j+1}}^{\epsilon_{j+1}}$ for any $j\in[1,l-1]$, where $\alpha_j$ are arrows in $E_{t_0}$ and $\epsilon_j\in\{1,-1\}$.
 
 For totally sign-skew-symmetric matrices, the second duality is generalized for mutation-acyclic case in \cite{liu2024acyclic} and in full generality in \cite{pan2023mutation}. To give this generalization, let us introduce some notations. 
 
 If $B$ is  totally sign-skew-symmetric matrix, then it follows from the facts that $\mu_k(-B)=-\mu_k(B)$ and $\mu_k(B^T) = (\mu_k(B))^T$ that $-B$, $B^T$ and $-B^T$ are totally sign-skew-symmetric. For a matrix pattern $(B_t,C_t,G_t)_{t\in \mathbb T_n}$ at the initial vertex $t_0$, we will denote $(\tilde B_t, \tilde C_t, \tilde G_t)_{t\in \mathbb T_n}$ the matrix pattern at the initial vertex $t_0$ such that $\tilde B_{t_0} = -B_{t_0}^T$, and we call it the {\em dual matrix pattern of the matrix pattern $(B_t,C_t,G_t)_{t\in \mathbb T_n}$}. Clearly $\tilde B_t= -B_t^T$ for any $t\in \mathbb T_n$.  Then we have the following lemma.
 \begin{lemma}\label{key}
 	With the notations above, for any $t\in \mathbb T_n$ and $i,j,k\in [1,n]$, we have that 
 	\begin{enumerate}
 	\item[(i)] the signs of the $k$-th columns of  $C_t$ and $\tilde C_t$ are the same;
 	\item[(ii)]the signs of the $k$-th rows of $G_t$ and $\tilde G_t$ are the same;
 	 	\item[(iii)] (Second duality) \begin{equation}\label{secdual} \tilde G_t^TC_t= I;\end{equation}
 	\item[(iv)]  the $i$-th column of $C_t$ is $\pm e_j$ if and only if the $i$-th column of $\tilde{C}_t$ is $\pm e_j$; and the $i$-th row of $G_t$ is $\pm e^T_j$ if and only if the $i$-th row of $\tilde{G}_t$ is $\pm e^T_j$.
 	\end{enumerate}
 \end{lemma}
 
 Lemma \ref{key} (i) was conjectured for acyclic sign-skew-symmetric matrices and proved for mutation-acyclic sign-skew-symmetric matrices in \cite{liu2024acyclic} and was proved in full generality in \cite{pan2023mutation}.  It was also proved in \cite{liu2024acyclic} that Lemma \ref{key} (ii), (iii), and (iv) are corollaries of (i). 
 
 Let $A=(a_{ij})_{n\times n}$ be a matrix. For any $j\in[1,n]$, we denote $[A]_+^{j\bullet}$ the $n\times n$ matrix such that the only nonzero row is the $j$-th row which is given by $([a_{k1}]_+, [a_{k2}]_+,\dots,[a_{kn}]_+)$, and we denote $J_j$ the $n\times n$ diagonal matrix whose diagonal entries are all $1$'s, except for $-1$ in the $j$-th position.
 
 Thanks to the sign-coherence of $C$-matrices and $G$-matrices, the mutations of $C$-matrices can be given in the matrix form as follows.  Let $(B_t,C_t^{t_0},G_t^{t_0})_{t\in \mathbb T_n}$ be a matrix pattern at the initial vertex $t_0$ and let $t\frac{k}{\quad\quad }t'$ be an arbitrary edge on $\mathbb T_n$, then $C_t^{t_0}$ and $C_{t'}^{t_0}$ are related by the following rule:
	\begin{equation}
	\label{cmut}  C^{t_0}_{t'} = C_t^{t_0}(J_k + [\varepsilon_k(C^{t_0}_t)B_{t}]_+^{k\bullet}). 
\end{equation}

Assume there is an edge $t_0\frac{j}{\quad\quad }t_1$ on $\mathbb T_n$.  Let $(B_t,C_t^{t_1},G_t^{t_1})_{t\in \mathbb T_n}$ be the matrix pattern at the initial vertex $t_1$. Then for any $t\in \mathbb T_n$, $C_t^{t_0}$ and $C_{t}^{t_1}$ are related by the following rule:
\begin{equation}\label{dualmut}  C^{t_1}_{t} = (J_j + [-\varepsilon_j(G^{t_0}_t)B_{t_0}]_+^{j\bullet})C_t^{t_0}.  \end{equation}

The rule (\ref{cmut}) follows immediately from the mutation rule (\ref{ccmut}) and the sign-coherence of $C$-matrices, and the rule (\ref{dualmut}) was proved by Nakanishi and Zelevinsky for skew-symmetrizable matrices in \cite{nakanishi2012duality}.  For totally sign-skew-symmetric matrices, it was shown that the rule (\ref{dualmut}) is a result of Lemma \ref{key} (i) in \cite{liu2024acyclic}. 

\subsection{Reddening sequences and maximal green sequences} In the rest of this subsection, we assume that $(B_t,C_t^{t_0},G_t^{t_0})_{t\in \mathbb T_n}$ is a matrix pattern at the initial vertex $t_0$. Thanks to the column sign-coherent property of $C$-matrices, we could distinguish the sign of each column index of a $C$-matrix in the following way. 
\begin{definition} For $t\in \mathbb T_n$,
	 a column index $j\in [1,n]$ is called {\em green} (respectively, {\em red}) of $C^{t_0}_t$  if $\varepsilon_j(C^{t_0}_t) =1$ (respectively, $\varepsilon_j(C^{t_0}_t) =-1$).
\end{definition}

Every column index of a $C$-matrix is either green or red, and all column indices of the initial $C$-matrix $C_{t_0}^{t_0}$ are green. 

\begin{definition}
	Let $(i_1,i_2,\dots,i_m)$ be a sequence such that $i_k\in [1,n]$ for all $1\leq k\leq m$, and let
	\[         (B_{t_k}, C^{t_0}_{t_k}, G^{t_0}_{t_k}) := \mu_{i_k}\dots\mu_{i_2}\mu_{i_1} (B_{t_0}, C^{t_0}_{t_0}, G^{t_0}_{t_0}),\quad k\in [1,m].                   \]
We say $(i_1,i_2,\dots,i_m)$ is a {\em reddening sequence of $B_{t_0}$} if all the column indices of $C^{t_0}_{t_m}$ are red. Denote $c^{t_0}_{i_{k+1}, t_{k}}$ the $i_{k+1}$-th column of $C^{t_0}_{t_{k}}$ for $0\leq k\leq m-1$. 	We call the sequence $c^{t_0}_{i_{k+1}, t_{k}}, 0\leq k\leq m-1$, the {\em associated sequence of $c$-vectors of the reddening sequence $(i_1,i_2,\dots,i_m)$.}
		We call the reddening sequence $(i_1,i_2,\dots,i_m)$ a {\em $r$-reddening sequence},
	if there are exactly $r$ $c$-vectors in $\{ c^{t_0}_{i_{k+1}, t_{k}} \,|\,  0\leq k\leq m-1\}$ having the negative sign. 
We call the $0$-reddening sequence a {\em maximal green sequence}.	
\end{definition}

Naturally, $r$-reddening sequences of $B_{t_0}$ one-to-one correspond to paths from $t_0$ to the vertex $t_0^-$, whose assigned $C$-matrix equals to $-I_n$ up to a permutation, admitting $r$ opposite arrows in the exchange graph.

Even though it is difficult to determine whether a totally sign-skew-symmetric matrix admits a reddening sequence or a maximal green sequence, there are lots of examples of skew-symmetrizable matrices having a reddening sequence or a maximal green sequence in literature.  Here we illustrate an example of a non-skew-symmetrizable matrix that has reddening sequences and maximal green sequences as follows.

\begin{example}
Let $B$ be the matrix given by
\[      B= \begin{pmatrix}
	0&1&2\\-1&0&1\\-1&-1&0
\end{pmatrix} .  \]
The matrix $B$ is totally sign-skew-symmetric and non-skew-symmetrizable. It is easy to check that 
\begin{enumerate}
\item[(i)] $(3,2,1)$ and $(2,3,1,2,1,2)$ are maximal green sequences of $B$, and
\item[(ii)] $(1,2,1,2,1,3,1,2)$ is a $2$-reddening sequence of $B$ with the associated sequence of $c$-vectors given by 
\[ \begin{pmatrix}
	1\\0\\0
\end{pmatrix},\quad    \begin{pmatrix}
1\\1\\0
\end{pmatrix},\quad  \begin{pmatrix}
0\\1\\0
\end{pmatrix},\quad   \begin{pmatrix}
-1\\0\\0
\end{pmatrix},\quad\begin{pmatrix}
0\\-1\\0
\end{pmatrix},\quad \begin{pmatrix}
0\\0\\1
\end{pmatrix},\quad \begin{pmatrix}
0\\1\\0
\end{pmatrix},\quad \begin{pmatrix}
1\\0\\0
\end{pmatrix}.     \]
\end{enumerate}
\end{example}

For an $n\times n$ totally sign-skew-symmetric matrix $B$, Br\"ustle, Dupont, and P$\mathrm{\acute{e}}$rotin proved that the length of a maximal green sequence is not strictly less than $n$ in \cite{brustle2014mgs}. Here we provide a proof by using the properties of the matrix pattern.

\begin{lemma}
	Let $(B_t,C_t^{t_0},G_t^{t_0})_{t\in \mathbb T_n}$ be a matrix pattern at the initial vertex $t_0$ and $(i_1,\dots,i_m)$ a  reddening sequence of $B_{t_0}$. Then $\{i_1,\dots, i_m\} = \{1,2,\dots,n\}$, \ie, every index $j\in [1,n]$ appears at least once in a reddening sequence.  In particular, $m\geq n$.
\end{lemma}
\begin{proof}  We depict the reddening sequence $(i_1,\dots,i_m)$  on $\mathbb T_n$ as follows.
		\[  t_0\frac{i_1}{\quad\quad}   t_1\frac{i_2}{\quad\quad}\;  \dots\;  \frac{i_m}{\quad\quad}   t_{m}. \]
	Let $(\tilde B_t,\tilde C_t^{t_0},\tilde G_t^{t_0})_{t\in \mathbb T_n}$ be the corresponding dual matrix pattern.
	If there is an index $j\in [1,n]$ such that $j\neq i_s$ for any $s\in [1,m]$, then the $j$-th column of $\tilde G_{t_m}^{t_0}$ is $e_j$ by the mutation rule (\ref{gmut}), here $e_j$ is the $j$-th column of the identity matrix. Thanks to the second duality (\ref{secdual}), we know the $j$-th row of $C_{t_m}^{t_0}$ is $e^T_j$, which contradicts to the fact that $C_{t_m}^{t_0}$ is non-positive.
\end{proof}

The following lemma is well-known, see Appendix 1 in \cite{brustle2014mgs}.

\begin{lemma}\label{ranktwo}
	Let $B = \begin{pmatrix}
	  0& a\\
	  -b&0
	\end{pmatrix}$ with $a,b >0$ and $ab\geq 4$. Then there is no maximal green sequence of $B$ beginning with the index $1$.
\end{lemma}

\subsection{Conjugation of reddening sequences}
In \cite{muller}, Muller introduced the conjugation of a reddening sequence for an arbitrary quiver and then showed that it induces a reddening sequence with respect to the initial quiver $Q$ to one with respect to the quiver $\mu_k(Q)$, which therefore leads to the mutation invariance of the existence of reddening sequences. Here we recall its definition with respect to a totally sign-skew-symmetric matrix.

Let $B$ be a totally sign-skew-symmetric matrix and $k$ belongs to the index set of $B$. For a reddening sequence $(i_1,\dots,i_m)$ of $B$ with the associated permutation $\sigma$, we call $(k,i_1,\dots,i_m,\sigma^{-1}(k))$ the \textit{conjugation} of $(i_1,\dots,i_m)$ in direction $k$.

\section{Conjugation of reddening sequences and the Rotation Lemma}\label{section: conjugation}
In this section, we prove the Conjugation Lemma and the Rotation Lemma for totally sign-skew-symmetric matrices. 

As in the previous section, for a matrix pattern  $(B_t,C^{t_0}_t,G^{t_0}_t)_{t\in \mathbb T_n}$, let $(\tilde B_t, \tilde C^{t_0}_t, \tilde G^{t_0}_t)_{t\in \mathbb T_n}$ be the dual matrix pattern at the initial vertex $t_0$ such that $\tilde B_{t_0} = -B_{t_0}^T$. For any permutation $\sigma \in S_n$, we denote $P_{\sigma}$ the permutation matrix associated to $\sigma$, more precisely, the $j$-th column of $P_{\sigma}$ is $e_{\sigma(j)}$, where $e_j$ is the $j$-th column of the identity matrix for all $j\in [1,n]$.

Before proving the main results, we need some preparations. The following lemma is easy.

\begin{lemma}\label{linearalg}
	Let $A=(a_{ij})_{n\times n}$ be a matrix, let $P_{\sigma}$ be the permutation matrix associated to a permutation $\sigma\in S_n$ and let $J_j$ be the $n\times n$ diagonal matrix whose diagonal entries are all $1'$s, except for $-1$ in the $j$-th position. Then the following equations hold
	\[ (J_j+[A]_+^{j\bullet} )^2=I,\quad        P_{\sigma}^TJ_jP_{\sigma} = J_{\sigma^{-1}(j)}, \quad \text{and}\quad   P_{\sigma}^T[A]_+^{j\bullet}P_{\sigma} = [P_{\sigma}^TAP_{\sigma}]_+^{\sigma^{-1}(j)\bullet} .         \]  
\end{lemma}

\begin{lemma}[\cite{ding2014nonnegative}]\label{ding}
	A real matrix and its inverse are non-negative if and only if it is the product of a diagonal matrix with all positive diagonal entries and a permutation matrix.
\end{lemma}

The following result, extending [Proposition 2.10,  \cite{brustle2014mgs}] and [Lemma 2.2.1, \cite{brustle2017semipicture}] to totally sign-skew-symmetric cluster algebras, implies that a reddenning sequence makes the $C$-matrix and $G$-matrix to the same negative permutation matrix.

\begin{lemma}\label{per}  Let $t\in \mathbb T_n$. The following results hold:
	\begin{enumerate}
		\item[(i)] $C^{t_0}_t$ is non-positive if and only if $\tilde G^{t_0}_t$ is non-positive;
		\item[(ii)] if $C^{t_0}_t$ is non-positive, then there is a permutation matrix $P$ such that 
		\[          C^{t_0}_t = G^{t_0}_t = \tilde G^{t_0}_t = \tilde C^{t_0}_t = -P.       \]
	\end{enumerate}
\end{lemma}

\begin{proof}
(i) Assume that $C^{t_0}_t$ is non-positive. By Lemma \ref{key} (iii), we have that $C^{t_0}_t(\tilde G^{t_0}_t)^T = I$. For any $j\in [1,n]$, there must be a negative entry in the $j$-th column of $(\tilde G^{t_0}_t)^T$. By the row-sign-coherence of $\tilde G^{t_0}_t$, every column vector of $(\tilde G^{t_0}_t)^T$ is sign-coherent and hence every column vector of  $(\tilde G^{t_0}_t)^T$ is  non-positive. So $\tilde G^{t_0}_t$ is non-positive. The converse proof is similar by the second duality $(\tilde G^{t_0}_t)^TC^{t_0}_t=I$
and the column-sign-coherence of $C^{t_0}_t$.

(ii) Since $C^{t_0}_t$ is non-positive, we have that $(\tilde G^{t_0}_t)^T$ is also non-positive by (i). Therefore $-C^{t_0}_t$ and its inverse $-(\tilde G^{\,t_0}_t)^T$ are non-negative matrices, and by Lemma \ref{ding}, $-C^{t_0}_t$ is a product of a diagonal matrix $D$ with all positive diagonal entries and a permutation matrix $P$. Since $-C^{t_0}_t = DP$ and  $-(\tilde G^{t_0}_t)^T=P^TD^{-1}$ are integer matrices, we have that $C^{t_0}_t=\tilde G^{t_0}_t=-P$. 

By Lemma \ref{key} (iv), we have that $G^{t_0}_t = \tilde C^{t_0}_t = -P.$
\end{proof}

Since the last $C$-matrix in a reddening sequence is non-negative, by Lemma \ref{per}, there is a permutation $\sigma$ such that the last  $C$-matrix is $-P_{\sigma}$. We call this permutation $\sigma$ the {\em associated permutation of the reddening sequence.}

Let $j\in [1,n]$. We define the two disjoint sets as follows:
\[  H_j^+:=    \{  C^{t_0}_t \;  |\;   \varepsilon_j(\tilde G^{t_0}_t)  =1,  \;t\in \mathbb T_n     \}, \quad \text{and} \quad H_j^- :=    \{  C^{t_0}_t \;  |\;   \varepsilon_j(\tilde G^{t_0}_t)  =-1  , \;t\in \mathbb T_n    \}.    \]
Clearly, $H^+_j \cup H_j^-$ is the set $\{  C^{t_0}_t ,\;t\in \mathbb T_n \} $ consisting of all $C$-matrices. We call $H^+_j$ and $H_j^-$ the {\em $j$-semispheres}.

The following result shows when two adjacent $C$-matrices in a matrix pattern belong to different $j$-semispheres.
\begin{lemma}\label{differhemis}
	Let $(B_t,C_t^{t_0},G_t^{t_0})_{t\in \mathbb T_n}$ be a matrix pattern at the initial vertex $t_0$. Let  $j\in[1,n]$.  Then for any edge $t\frac{k}{\quad\quad }t'$ on $\mathbb T_n$, the following statements hold:
	\begin{enumerate}
		\item[(i)] $C^{t_0}_t \in H_j^+$ and $C^{t_0}_{t'} \in H_j^-$ if and only if $c^{t_0}_{k,t}=e_j$ if and only if $c^{t_0}_{k,t'}=-e_j$;
		\item[(ii)] $C^{t_0}_t \in H_j^-$ and $C^{t_0}_{t'} \in H_j^+$ if and only if $c^{t_0}_{k,t}=-e_j$ if and only if $c^{t_0}_{k,t'}=e_j$;
		\item[(iii)] $C^{t_0}_t$ and $C^{t_0}_{t'}$ are in the different $j$-hemispheres if and only if $c^{t_0}_{k,t}=\pm e_j$.
	\end{enumerate}
\end{lemma}
\begin{proof}
It follows directly from (\ref{cmut}) that $c^{t_0}_{k,t}=\pm e_j$ if and only if $c^{t_0}_{k,t'}=\mp e_j$. It follows from the second duality (\ref{secdual}) that $c^{t_0}_{k,t}=\pm e_j$  if and only if  the $j$-th row of $\tilde G^{t_0}_t$ is $\pm e^T_k$. Hence if $c^{t_0}_{k,t}=\pm e_j$ and $c^{t_0}_{k,t'}=\mp e_j$, then $C_t^{t_0}\in H_j^{\pm}$ and  $C_{t'}^{t_0}\in H_j^{\mp}$ . 

Conversely, if $C_t^{t_0}\in H_j^{+}$ and $C^{t_0}_{t'} \in H_j^-$, this implies that $\varepsilon_j(\tilde G^{t_0}_t)=1$ and $\varepsilon_j(\tilde G^{t_0}_{t'})=-1$. Note that $\tilde G^{t_0}_t$ and $\tilde G^{t_0}_{t'}$ only differs in the $k$-th columns. Thanks to the  unimodularity  and the  row-sign-coherence of $G$-matrices, the $j$-th rows of $\tilde G^{t_0}_t$ and $\tilde G^{t_0}_{t'}$ are $e^T_k$ and $-e^T_k$ respectively, and thus $c^{t_0}_{k,t}=e_j$. Similarly, if $C_t^{t_0}\in H_j^{-}$ and $C^{t_0}_{t'} \in H_j^+$, then  $c^{t_0}_{k,t}=-e_j$.
\end{proof}

As a result of Lemma \ref{differhemis}, we show that in a reddening sequence we should mutate at the $c$-vector $e_j$ once more than that at the $c$-vector $-e_j$ for any $j\in[1,n]$.

\begin{corollary}\label{lll}
	Let $(B_t,C_t^{t_0},G_t^{t_0})_{t\in \mathbb T_n}$ be a matrix pattern at the initial vertex $t_0$ and $(i_1,\dots,i_m)$ a reddening sequence of $B_{t_0}$ with the associated sequence of $c$-vectors $c_0,\dots,c_{m-1}$. Let $j\in [1,n]$. Then the number of $c$-vectors equal to $-e_j$ in $c_0,\dots,c_{m-1}$ is $l$ if and only if	 the number of $c$-vectors equal to $e_j$ in $c_0,\dots,c_{m-1}$ is $l+1$. 
	
	If in addition, $(i_1,\dots,i_m)$ is a maximal green sequence, then for any $j\in [1,n]$, $e_j$ appears exactly once in the associated sequence of $c$-vectors $c_0,\dots,c_{m-1}$. 
\end{corollary}
\begin{proof}
Let us depict the reddening sequence $(i_1,\dots,i_m)$ on $\mathbb T_n$ as follows.
	\[  
(B_{t_0} ,C^{t_0}_{t_0},G^{t_0}_{t_0})
\frac{i_1}{\quad\quad} (B_{t_1} ,C^{t_0}_{t_1},G^{t_0}_{t_1})\frac{i_2}{\quad\quad}\;  \dots\;  \frac{i_m}{\quad\quad}  (B_{t_m} ,C^{t_0}_{t_m},G^{t_0}_{t_m}).\]
In this case, $c_k= c^{t_0}_{i_{k+1},t_{k}}$ for $0\leq k\leq m-1$.
 
Let $j\in [1,n]$.  Thanks to Lemma \ref{differhemis}, for $0\leq k\leq m-1$, the adjacent $C$-matrices $C_{t_k}^{t_0}$ and $C_{t_{k+1}}^{t_0}$ belong to the same $j$-hemisphere if and only if $c_k \neq  \pm e_j$, and if $c_k = e_j$, then $C_{t_k}^{t_0}\in H_j^+$ and $C_{t_{k+1}}^{t_0}\in H_k^-$, and if $c_k = -e_j$, then $C_{t_k}^{t_0}\in H_j^-$ and $C_{t_{k+1}}^{t_0}\in H_k^+$.  Since $C^{t_0}_{t_0}\in H_j^+$ and $C^{t_0}_{t_m}\in H_j^-$, the number of $c$-vectors equal to $e_j$ is exactly one more than  the number of $c$-vectors equal to $e_j$ in $c_0,\dots,c_{m-1}$. 
\end{proof}

The following lemma shows how the signs of $c$-vectors are changed under the changing of the initial seed. It turns out that all the signs remain but those of $c$-vectors equal to $\pm e_j$ when the changing is realized as the mutation in direction $j$.
\begin{lemma}\label{kk3}
	Let $(B_t,C_t^{t_0},G_t^{t_0})_{t\in \mathbb T_n}$ be a matrix pattern at the initial vertex $t_0$, let $t_0\frac{j}{\quad\;}t_1$ be the edge labeled by $j$ in $\mathbb T_n$, and let $(B_t,C_t^{t_1},G_t^{t_1})_{t\in \mathbb T_n}$ be the matrix pattern at the initial vertex $t_1$. For any $t\in \mathbb T_n$ and $k\in [1,n]$, let $c^{t_0}_{k,t}$ and $c^{t_1}_{k,t}$ be the $k$-th columns of $C^{t_0}_{t}$ and $C^{t_1}_{t}$, respectively. Then 
	\begin{enumerate}
	\item[(i)]  $c^{t_0}_{k,t} =  e_j$ if and only if $c^{t_1}_{k,t} =  -e_j$ ;
	\item[(ii)] $c^{t_0}_{k,t} =  -e_j$ if and only if $c^{t_1}_{k,t} =  e_j$;
	\item[(iii)] $\sgn(c^{t_0}_{k,t}) = - \sgn(c^{t_1}_{k,t})$ if and only if $c^{t_0}_{k,t} = \pm e_j$.
\end{enumerate}
\end{lemma}
\begin{proof}  
    Clearly  by (\ref{dualmut}), $c^{t_0}_{k,t} =  \pm e_j  \Longleftrightarrow c^{t_1}_{k,t} =  \mp e_j$. This proves (i) and (ii).
    
	The equation (\ref{dualmut}) also shows that $C^{t_0}_{t}$ and $C^{t_1}_{t}$ only differ in the $j$-th rows, and hence  $c^{t_0}_{k,t}$ and $c^{t_1}_{k,t}$ only differ in the $j$-th coordinates. If there exists some $i\neq j$ such that the $i$-th coordinate of 
	$c^{t_0}_{k,t}$ is nonzero, then $\sgn(c^{t_0}_{k,t}) = \sgn(c^{t_1}_{k,t})$ by the sign-coherence of $c$-vectors. Therefore $\sgn(c^{t_0}_{k,t}) = -\sgn(c^{t_1}_{k,t})$ implies that every $i$-th coordinate of 
	$c^{t_0}_{k,t}$ is zero for each $i\neq j$ and then  $c^{t_0}_{k,t} = \pm e_j$ by the unimodularity. The converse proof of (iii) follows from (i) and (ii).
\end{proof}

Now we are ready to analyze how conjugation of a reddening sequence acts on signs of the corresponding $c$-vector sequence.
\begin{theorem}[Conjugation Lemma]\label{conj}
  Let $(B_t,C_t^{t_0},G_t^{t_0})_{t\in \mathbb T_n}$ be a matrix pattern at the initial vertex $t_0$, $(i_1,\dots,i_m)$ be an $r$-reddening sequence of $B_{t_0}$ with the associated permutation $\sigma$ and $t_0\frac{j}{\quad\;}t'$ be the edge labeled by $j$ in $\mathbb T_n$.  Then the conjugation of $(i_1,\dots,i_m)$ in direction $j$ is an $(r+1)$-reddening sequence of $B_{t'}=\mu_{j}(B_{t_0})$ with the associated permutation $\sigma$.  
\end{theorem}
\begin{proof}
    Let $\sigma^{-1}(j)=k$.	We depict the path on $\mathbb T_n$ associated with the conjugation of $(i_1,\dots,i_m)$ in direction $j$ as follows.
	\[  t'\frac{j}{\quad\quad}   t_0\frac{i_1}{\quad\quad}\;  \dots\;  \frac{i_m}{\quad\quad}   t_{m}\frac{ k}{\quad\quad }  t_{m+1}  \]
	
	For the matrix patterns $(B_t,C_t^{t_0},G_t^{t_0})_{t\in \mathbb T_n}$ and $(B_t,C_t^{t'},G_t^{t'})_{t\in \mathbb T_n}$, we respectively depict the sequences of matrices as follows. They are
		\[  
		(B_{t'} ,C^{t_0}_{t'},G^{t_0}_{t'})
		 \frac{j}{\quad\quad} (B_{t_0} ,C^{t_0}_{t_0},G^{t_0}_{t_0})\frac{i_1}{\quad\quad}\;  \dots\;  \frac{i_m}{\quad\quad}  (B_{t_m} ,C^{t_0}_{t_m},G^{t_0}_{t_m})\frac{k}{\quad\quad } (B_{t_{m+1}} ,C^{t_0}_{t_{m+1}},G^{t_{0}}_{t_{m+1}})\]
	and 
	\[  
      (B_{t'} ,C^{t'}_{t'},G^{t'}_{t'})
       \frac{j}{\quad\quad} (B_{t_0} ,C^{t'}_{t_0},G^{t'}_{t_0})\frac{i_1}{\quad\quad}\;  \dots\;  \frac{i_m}{\quad\quad}  (B_{t_m} ,C^{t'}_{t_m},G^{t'}_{t_m})\frac{k}{\quad\quad } (B_{t_{m+1}} ,C^{t'}_{t_{m+1}},G^{t'}_{t_{m+1}}),\]
	here $C^{t_0}_{t_0} = C^{t'}_{t'}= I$. 
	
	By Lemma \ref{per}, $C^{t_0}_{t_m} =G^{t_0}_{t_m}= -P_{\sigma}$, and we need to show that  $C^{t'}_{t_{m+1}} = -P_{\sigma}$.  By the first duality (\ref{firstduality}), we have that
   \begin{equation}   B_{t_m} = (G^{t_0}_{t_m})^{-1}B_{t_0} C^{t_0}_{t_m} =    P_{\sigma}^TB_{t_0}P_{\sigma}.  \end{equation}
	
	 Since $C^{t_0}_{t_m} =G^{t_0}_{t_m}= -P_{\sigma}$ are non-positive matrices, then $\varepsilon_i(C^{t_0}_{t_m} ) =  \varepsilon_i(G^{t_0}_{t_m} ) =-1$ for every $i\in [1,n]$.
	 
	 Since the $k$-th column of $C^{t_0}_{t_{m}}$ is $-e_j$, then the $k$-th column of $C^{t'}_{t_{m}}$ is $e_j$ by Lemma \ref{kk3} (ii) and hence $\varepsilon_{k}(C^{t'}_{t_{m}} ) = 1$. 
  
  Combining (\ref{cmut}), (\ref{dualmut}) and Lemma \ref{linearalg}, we have that
        \begin{equation} \begin{split}
            C^{t'}_{t_{m+1}}  & =C^{t'}_{t_{m}} (J_k+[\varepsilon_{k}                      (C^{t'}_{t_{m}})B_{t_m}]_+^{k\bullet})\\
	                          &=(J_j + [-\varepsilon_{j}(G^{t_0}_{t_{m}} )B_{t_0}]_+^{j\bullet})	C^{t_0}_{t_{m}}  (J_k  +  [\varepsilon_{k}(C^{t'}_{t_{m}} )B_{t_m}]_+^{k\bullet})\\  
	                          &=-(J_j + [B_{t_0}]_+^{j\bullet})	P_{\sigma} (J_{\sigma^{-1}(j)}  +  [P_{\sigma}^TB_{t_0}P_{\sigma}]_+^{\sigma^{-1}(j)\bullet})\\  
	                          &=  -(J_j + [B_{t_0}]_+^{j\bullet})	P_{\sigma} (P_{\sigma}^T J_{j} P_{\sigma}  +  P_{\sigma}^T[B_{t_0}]_+^{j\bullet}P_{\sigma})\\
	                         &=   -(J_j + [B_{t_0}]_+^{j\bullet})	P_{\sigma} P_{\sigma}^T (J_j    +              [B_{t_0}]_+^{j\bullet})P_{\sigma}\\
	                         &=-P_{\sigma}.
	             	\end{split}
        \end{equation}               
This proves $(j,i_1,\dots,i_m,\sigma^{-1}(j))$
is a reddening sequence of $B_{t'}=\mu_j(B)$ with the associated permutation $\sigma$.

 Denote $c_0,\dots,c_{m-1}$ and $c', c'_0, \dots, c'_{m-1}, c'_m$ the associated sequences of $c$-vectors of the reddening sequences of $B_{t_0}$ and $B_{t'}$ respectively. We have proved $c'_m=e_j=c'$. Assume that  the number of $c$-vectors equal to $e_j$ in $c_0,\dots,c_{m-1}$ is $l$, then the number of $c$-vectors equal to $-e_j$ in $c_0,\dots,c_{m-1}$ is $l-1$, thanks to Corollary \ref{lll}.  By Lemma \ref{kk3}, for each $0\leq i\leq m-1$, $c_i$ and $c'_i$ have the following relations:
 \begin{enumerate}
 	\item[(1)] $c_i=\pm e_j \Longleftrightarrow c'_i= \mp e_j$, thus $e_j$ appears exactly $l-1$ times in $c'_0,\dots,c'_{m-1}$, and $-e_j$ appears exactly $l$ times in $c'_0,\dots,c'_{m-1}$; 
 	\item[(2)] $c_i\neq \pm e_j \Longleftrightarrow \sgn(c_i)= \sgn(c'_i)$.
 \end{enumerate}
 This proves $(j,i_1,\dots,i_m,\sigma^{-1}(j))$
 is an $(r+1)$-reddening sequence of $B_{t'}$.
\end{proof}

An immediate corollary of the Conjugation Lemma shows the existence of reddening sequences is mutation invariant.

\begin{corollary}
   Let $(B_t,C_t^{t_0},G_t^{t_0})_{t\in \mathbb T_n}$ be a matrix pattern at the initial vertex $t_0$. If $B_{t_0}$ admits a reddening sequence, then so does $B_t$ for any $t\in \mathbb T_n$. 
\end{corollary}

In particular, when $j$ equals $i_1$ in Theorem \ref{conj}, we have the following rotation lemma.

\begin{corollary}[Rotation Lemma]\label{rotationlemma}
Let $(B_t,C_t^{t_0},G_t^{t_0})_{t\in \mathbb T_n}$ be a matrix pattern at the initial vertex $t_0$ and $t_0\frac{i_1}{\quad\;}t_1$ be an edge labeled by $i_1$ in $\mathbb T_n$. Then the sequence $(i_1,\dots,i_m)$ is an $r$-reddening sequence of $B_{t_0}$ with the associated permutation $\sigma$ if and only if the sequence 
\[     (i_2,\dots,i_m,\sigma^{-1}(i_1))                \]
is an $r$-reddening sequence of $B_{t_1}=\mu_{i_1}(B_{t_0})$ with the associated permutation $\sigma$.
 
 In particular, $(i_1,\dots,i_m)$ is a maximal green sequence of $B_{t_0}$ if and only if 
 \[(i_2,\dots,i_m,\sigma^{-1}(i_1))\] is a maximal green sequence of $\mu_{i_1}(B_{t_0})$.
\end{corollary}

\section{Conjugation difference}\label{section: potential}
\begin{definition}
    Let $(B_t,C_t^{t_0},G_t^{t_0})_{t\in \mathbb T_n}$ be a matrix pattern at the initial vertex $t_0$.
    \begin{enumerate}
        \item[$(1)$]  A sequence $(i_1,\dots,i_m)$ is called a \emph{greening sequence} (or a \emph{cycle sequence}) of $B_{t_0}$ if 
        \[\mu_{i_m}\dots\mu_{i_1}(C^{t_0}_{t_0})=P,\]
        where $P$ is a permutation matrix, and the permutation corresponding to $P$ is called the {\emph associated permutation of the greening sequence}.

        \item[$(2)$] An \emph{$r$-greening sequence} is a greening sequence with $r$ red mutations.

        \item [$(3)$] The \emph{conjugation} of a greening sequence $(i_1,\dots,i_m)$ with the associated permutation $\sigma$ in direction $j$ equals $$(j,i_1,\dots,i_m,\sigma^{-1}(j))$$
        for any $j\in[1,n]$.
    \end{enumerate}    
\end{definition}

The following result is a corollary of Lemma \ref{differhemis}, and its proof is similar to Corollary \ref{lll}.
\begin{lemma}
    Let $(B_t,C_t^{t_0},G_t^{t_0})_{t\in \mathbb T_n}$ be a matrix pattern at the initial vertex $t_0$, and let $(i_1,\dots,i_m)$ be a greening sequence with the associated sequence of $c$-vectors $c_0,\dots,c_{m-1}$. Let $j\in [1,n]$. Then the number of $c$-vectors equal to $-e_j$ in $c_0,\dots,c_{m-1}$ is the same as that of $c$-vectors equal to $e_j$ in $c_0,\dots,c_{m-1}$.
\end{lemma}

Similar to reddening sequences, the conjugation or rotation of a greening sequence results in another greening sequence.

\begin{proposition}
    Let $(B_t,C_t^{t_0},G_t^{t_0})_{t\in \mathbb T_n}$ be a matrix pattern at the initial vertex $t_0$, let $(i_1,\dots,i_m)$ be an $r$-greening sequence of $B_{t_0}$ with the associated permutation $\sigma$, and let $t_0\frac{j}{\quad\;}t'$ be the edge labeled by $j$ in $\mathbb T_n$. Then the conjugation of $(i_1,\dots,i_m)$ in direction $j$ is an $(r+1)$-greening sequence of $B_{t'}=\mu_{j}(B_{t_0})$ with the associated permutation $\sigma$.
\end{proposition}
\begin{proof}
 The proof is similar to that of Theorem \ref{conj}.
 \end{proof}
\begin{proposition}\label{rotation lemma of greening seqs}
    Let $(B_t,C_t^{t_0},G_t^{t_0})_{t\in \mathbb T_n}$ be a matrix pattern at the initial vertex $t_0$ and $t_0\frac{i_1}{\quad\;}t_1$ be an edge labeled by $i_1$ in $\mathbb T_n$. Then the sequence $(i_1,\dots,i_m)$ is an $r$-greening sequence of $B_{t_0}$ with the associated permutation $\sigma$ if and only if the sequence 
    \[(i_2,\dots,i_m,\sigma^{-1}(i_1)) \]
    is an $r$-greening sequence of $B_{t_1}=\mu_{i_1}(B_{t_0})$ with the associated permutation $\sigma$.
\end{proposition}

For a vertex $t\in\mathbb T_n$, denoted by $t^-$  a vertex in $\mathbb T_n$ such that $C^{t}_{t^-}=-P$, where $P$ is a permutation matrix. 

Let $(B_t,C_t^{t_0},G_t^{t_0})_{t\in \mathbb T_n}$ be a matrix pattern at the initial vertex $t_0$ and $t\in\mathbb T_n$. Denote by $(i_1,\dots,i_m)$ the sequence of labels of edges passed through by the path from $t$ to $t_0$ in $\mathbb T_n$. Choose $t_0^-$ and $t^-$ such that $(\sigma^{-1}(i_1),\dots,\sigma^{-1}(i_m))$ equals the sequence of labels of edges passed through by the path from $t^-$ to $t_0^-$ in $\mathbb T_n$ and $\sigma(C^{t_0}_{t^-_0})=I_n$ for some permutation $\sigma$, where $\sigma$ acts on matrices by permuting indices. Therefore, the choice of $t^-_0$ uniquely determines $t^-$ and $\sigma$.

\begin{definition}
    In the above settings, denote by $\varphi^{t_0}_t$ the \emph{conjugation difference} of $t$ with respect to $t_0$, which equals the difference of the number of red mutations in $(i_1,\dots,i_m)$ from $t$ to $t_0$ and that in $(\sigma^{-1}(i_1),\dots,\sigma^{-1}(i_m))$ from $t^-$ to $t_0^-$.
\end{definition}
Different choices of $t^-_0$, say $t_1$ and $t_2$,  induces different choices of $t^-$, say $t_3$ and $t_4$ respectively, and different permutations $\sigma_1$ and $\sigma_2$ such that $\sigma_1(C_{t_1}^{t_0})=\sigma_2(C_{t_2}^{t_0})=I_n$. Hence $\sigma_1(\mu_{\sigma_1^{-1}(i_s)}\cdots\mu_{\sigma_1^{-1}(i_m)}(C_{t_1}^{t_0}))=\sigma_2(\mu_{\sigma_2^{-1}(i_s)}\cdots\mu_{\sigma_2^{-1}(i_m)}(C_{t_2}^{t_0}))$ for any $s\in[1,m]$. In particular, the numbers of red mutations in $(\sigma_1^{-1}(i_1),\dots,\sigma_1^{-1}(i_m))$ from $t_3$ to $t_1$ equals that in $(\sigma_2^{-1}(i_1),\dots,\sigma_2^{-1}(i_m))$ from $t_4$ to $t_2$. Hence $\varphi^{t_0}_t$ does not depend on the choice of $t^-_0$.

Recall that a maximal green sequence of $B_t$ is a mutation sequence from $t$ to $t^-$ admitting no red mutations in the matrix pattern with the initial vertex $t$. However, such mutation sequence may contain red mutations when we deal with another matrix pattern. The following result claims the number of red mutations is exactly $\varphi_t^{t_0}$, which only depends on $t$ and the choice of the initial vertex $t_0$.

\begin{theorem}\label{thm of conjugation difference}
    Let $(B_t, C_t^{t_0}, G_t^{t_0})_{t\in \mathbb T_n}$ be a matrix pattern at the initial vertex $t_0$, $t,t'\in\mathbb T_n$ and $(i_1,\dots,i_m)$ be the sequence of indices passed through by the path from $t$ to $t^-$ with $r$ red mutations. 
    \begin{enumerate}
        \item[$(1)$] $\varphi^{t_0}_t\leqslant r$.
        \item[$(2)$] $\varphi^{t_0}_t=\varphi^{t_0}_{t'}$ if $C_t^{t_0}=C_{t'}^{t_0}$ up to a simultaneous permutation of rows and columns.

        \item[$(3)$] $(i_1,\dots,i_m)$ is an $(r-\varphi^{t_0}_t)$-reddening sequence of $B_{t}$. In particular, $(i_1,\dots,i_m)$ is a maximal green sequence of $B_{t}$ if and only if $\varphi^{t_0}_t=r$.
    \end{enumerate}
\end{theorem}
\begin{proof}
    Assume $t,t',t^-\in\mathbb T_n$ satisfying $\sigma(C_{t}^{t_0})=C_{t'}^{t_0}$ for a permutation $\sigma$ and $\sigma'(C_{t^-}^{t_0})=-I_n$ for a permutation $\sigma'$. Let $(j_1, \dots, j_l)$ be the sequence of indices passed through by the path from $t$ to $t_0$. Then $(\sigma(j_1),\dots,\sigma(j_l)$ is the sequence of indices passed through by the path from $t'$ to a vertex $t_1$ with $\sigma^{-1}(C_{t_1}^{t_0})=I_n$ and $((\sigma')^{-1}(j_1),\dots,(\sigma')^{-1}(j_l)$ is the sequence of indices passed through by the path from $t^-$ to $t_0^-$. Denote by $(k_1, \dots, k_s)$ the sequence of indices passed through by the path from $t_1$ to $t$. 

    In a matrix pattern $(B_t,C_t^{t_1},G_t^{t_1})_{t\in \mathbb T_n}$ at the initial vertex $t_1$, assume the number of red mutations in $((\sigma')^{-1}(j_1),\dots,(\sigma')^{-1}(j_l))$ from $t^-$ to $t_0^-$,that in $(k_1,\dots,k_s)$ from $t_1$ to $t$ and that in $(j_1,\dots,j_l)$ from $t$ to $t_0$ equals $a$, $b$ and $a''$ respectively. Also note that the number of red mutations in $(i_1,\dots,i_m)$ from $t$ to $t^-$ equals $r$ since $\sigma^{-1}(C_{t_1}^{t_0})=I_n$.

    In a matrix pattern $(B_t,C_t^{t'},G_t^{t'})_{t\in \mathbb T_n}$ at the initial vertex $t'$, assume the number of red mutations in $(i_1,\dots,i_m)$ from $t$ to $t^-$, that in $(\sigma(j_1),\dots,\sigma(j_l))$ from $t'$ to $t_1$ and that in $(k_1,\dots,k_s)$ from $t_1$ to $t$ equals $r'$, $a'$ and $b'$ respectively. 
    
    Then $(\sigma(j_1),\dots,\sigma(j_l),k_1,\dots,k_s)$ is a $(a'+b')$-greening sequence of $B_{t'}$ and \[(\sigma(j_1),\dots,\sigma(j_l),k_1,\dots,k_s, i_1,\dots,i_m)\] is a $(a'+b'+r')$-reddening sequence of $B_{t'}$, while $(k_1,\dots,k_s,j_1,\dots,j_l)$ is a $(a''+b)$-greening sequence of $B_{t_1}$ and $(k_1,\dots,k_s, i_1,\dots,i_m,(\sigma')^{-1}(j_1),\dots,(\sigma')^{-1}(j_l))$ is a $(a+b+r)$-reddening sequence of $B_{t_1}$.

    Following from Corollary \ref{rotationlemma} and Proposition \ref{rotation lemma of greening seqs}, we have the following equations
    \[a'+b'=a''+b,\qquad\text{and}\qquad a'+b'+r'=a+b+r.\]
    Therefore, $r'=a+b+r-(a'+b')=a+b+r-(a''+b)=r-(a''-a)$, which equals $r-\varphi_t^{t_0}$ as $\sigma^{-1}(C_{t_1}^{t_0})=I_n$ and hence induces (1) and (3). (2) follows from (3) and the fact that a $p$-reddening sequence of $B_t$ induces a $p$-reddening sequence of $B_{t'}$ via the permutation $\sigma$.
\end{proof}

According to Theorem \ref{thm of conjugation difference} (2), the conjugation difference $\varphi_t^{t_0}$ is also well-defined for the exchange graph. In this case, $\varphi_t^{t_0}$ does not depend on the choices of the path from $t$ to $t_0$. 

Then we may regard a maximal green sequence of $B_t$ as a sequence of labels of edges passed through by a path from $t$ to $t^-$ with the theoretically minimal number $\varphi_t^{t_0}$ of red mutations, which equals $0$ when $t$ is the initial vertex but may be positive in general. This allows us to talk about maximal green sequences of any exchange matrix in the mutation equivalent class in one fixed matrix pattern.

\begin{corollary}
    Let $(B_t,C_t^{t_0},G_t^{t_0})_{t\in \mathbb T_n}$ be a matrix pattern at the initial vertex $t_0$ with corresponding exchange graph $E_{t_0}$. A maximal green sequence of $B_t$ is a sequence of mutations corresponding to a path from $t$ to $t^-$ in the exchange graph $E_{t_0}$ with exactly $\varphi_t^{t_0}$ opposite arrows.
\end{corollary}

\begin{remark}
    We may also explain the Conjugation Lemma and hence the Rotation Lemma via conjugation difference as follows: let $t_0\frac{j}{\quad\;}t'$ be an edge labeled by $j$ in $\mathbb T_n$, then by definition, the sum of the number of red mutations of two extra mutations added by the conjugation in direction $j$ and $\varphi_{t'}^{t_0}$ equals $1$.
\end{remark}

\section{Target before Source Theorem}\label{section: target before source}
In this section, we prove Target before Source Theorem for acyclic sign-skew-symmetric matrices and Target before Source Theorem of $c$-vectors for arbitrary totally sign-skew-symmetric matrices.  
 
\begin{lemma}\label{qq}
Let $(B_t,C_t^{t_0},G_t^{t_0})_{t\in \mathbb T_n}$ be a matrix pattern at the initial vertex $t_0$ and  $j\in [1,n]$ a source of $B_{t_0} = (b_{ij})_{n\times n }$. Let $t_0\frac{i_1}{\quad\quad}   t_1\frac{i_2}{\quad\quad}\;  \dots\;  \frac{i_m}{\quad\quad}   t_{m}$ be a path on $\mathbb T_n$ such that $i_p\neq j$ for $p\in [1,m]$. Then the $j$-th row of $C^{t_0}_{t_m}$ are $e^T_j$, and the $j$-th column of $C^{t_0}_{t_m}$ is $e_j$. 
\end{lemma}
\begin{proof}
Let $(\tilde B_t, \tilde C^{t_0}_t, \tilde G^{t_0}_t)_{t\in \mathbb T_n}$ be the dual matrix pattern of $(B_t,C_t^{t_0},G_t^{t_0})_{t\in \mathbb T_n}$. It follows from (\ref{gmut}) that the $j$-th column of $\tilde{G}_t^{t_0}$ is $e_j$ and hence $j$-th row of $C_t^{t_0}$ is $e^T_j$, thanks to the second duality (\ref{secdual}). We prove that the $j$-th column of $C^{t_0}_{t_m}$ is $e_j$ by induction on $m$. Recall that $j$ is a source of $B$ if and only if $b_{ij}\leq 0$ for any $i\in [1,n]$.

If $m=1$, then $C^{t_0}_{t_1} = J_{i_1} +  [B_{t_0}]_+^{i_1\bullet}$. Since $i_1\neq j$ and $b_{i_1j}\leq 0$, then the $j$-th column of $C^{t_0}_{t_1}$ is $e_j$.

Let $k\in [1,n]$ and $k\neq j$, and let $t_0\frac{i_1}{\quad\quad}   t_1\frac{i_2}{\quad\quad}\;  \dots\;  \frac{i_m}{\quad\quad}   t_{m}\frac{k}{\quad\quad}   t_{m+1}$ be the path on $\mathbb T_n$.  
Let $B_{t_m}= (b_{ij,t_m})_{n\times n}$. Since $C^{t_0}_{t_{m+1}} = C^{t_0}_{t_m}(J_{k} +  [\varepsilon_k(C^{t_0}_{t_m})B_{t_m}]_+^{k\bullet})$, we have that \[c_{j,t_{m+1}}^{t_0} = c_{j,t_m}^{t_0} + [\varepsilon_k(C^{t_0}_{t_m})b_{kj,t_m}]_+\cdot c_{k,t_m}^{t_0}. \]
By the first duality (\ref{firstduality}) and the second duality (\ref{secdual}), we have that
\[    B_{t_m} = (G_{t_m}^{t_0})^{-1}B_{t_0}C^{t_0}_{t_m} = (\tilde{C}_{t_m}^{t_0})^{T}B_{t_0}C^{t_0}_{t_m}.         \]
Therefore \[ b_{kj,t_m} = (\tilde c_{k,t_m}^{ \,t_0})^T B_{t_0}c^{t_0}_{j,t_m} = (\tilde c_{k,t_m}^{ \,t_0})^T B_{t_0}e_j =  (\tilde c_{k,t_m}^{ \,t_0})^T \mathbf{b}_j, \]
where $\mathbf b_j$ is the $j$-th column of $B_{t_0}$ and since $j$ is a source, hence $\mathbf b_j$ is non-positive. By Lemma \ref{key} (i), $\sgn (\tilde c_{k,t_m}^{ \,t_0})= \varepsilon_k(C^{t_0}_{t_m})$. Therefore
\[     \varepsilon_k(C^{t_0}_{t_m})  b_{kj,t_m} =   \varepsilon_k(C^{t_0}_{t_m}) \cdot (\tilde c_{k,t_m}^{ \,t_0})^T \mathbf{b}_j    =    \sgn (\tilde c_{k,t_m}^{ \,t_0})  \cdot   (\tilde c_{k,t_m}^{ \,t_0})^T \mathbf{b}_j \leq 0,    \]
and thus $c_{j,t_{m+1}}^{t_0} = c_{j,t_m}^{t_0} + [\varepsilon_k(C^{t_0}_{t_m})b_{kj,t_m}]_+\cdot c_{k,t_m}^{t_0} = c_{j,t_m}^{t_0} = e_j$. This finishes the proof.
\end{proof}

The following lemma was proved in \cite{brustle2017semipicture} using categoriﬁcation of cluster algebras. Here we provide a proof that works for acyclic sign-skew-symmetric matrices.

\begin{lemma}\label{xxx}
	 Let $(B_t,C_t^{t_0},G_t^{t_0})_{t\in \mathbb T_n}$ be a matrix pattern at the initial vertex $t_0$ such that $B_{t_0} = (b_{ij})_{n\times n }$ is acyclic.  Suppose that $b_{ji}>0$.  Let $t_0\frac{i_1}{\quad\quad}   t_1\frac{i_2}{\quad\quad}\;  \dots\;  \frac{i_m}{\quad\quad}   t_{m}$ be a path on $\mathbb T_n$ such that $i_p\neq i, j$ for $p\in [1,m]$ and let  $B_{t_m} = (b_{ij,t_m})$. Then $b_{ji,t_m} \geq b_{ji}$ and $b_{ij,t_m} \leq b_{ij}$.  In particular,  $-b_{ji,t_m}b_{ij,t_m}\geq -b_{ji}b_{ij}$.
 \end{lemma}
\begin{proof}
	We may, without loss of generality, assume that $j<i$ and $(1,2,\dots,n)$ is an admissible sequence of sources of $B_{t_0}$. Let
\[ t_{*}\frac{j-1}{\quad\quad}\;  \dots\; \frac{2}{\quad\quad}   t_{-1} \frac{1}{\quad\quad}   t_0\frac{i_1}{\quad\quad}   t_1\frac{i_2}{\quad\quad}\;  \dots\;  \frac{i_m}{\quad\quad}   t_{m}\] be the path on $\mathbb T_n$.
Clearly, $j$ is a source of $B_{t_{*}}$  and $b_{ij,t_{*}} = b_{ij}$ and $b_{ji,t_{*}} = b_{ji}$. Let $(B_t,C_t^{t_*},G_t^{t_*})_{t\in \mathbb T_n}$ be the matrix pattern at the initial vertex $t_*$, and $(\tilde B_t, \tilde C_t^{t_*},\tilde G_t^{t_*})_{t\in \mathbb T_n}$ its dual matrix pattern. Then 
\[     B_{t_m} =       (\tilde C_{t_m}^{t_*})^T B_{t_*}     C_t^{t_*}\quad \Longrightarrow  \quad b_{ji,t_m} =  (\tilde c^{\,t_*}_{j,t_m})^T B_{t_*} c^{t_*}_{i,t_m},\quad \text{and}\quad  b_{ij,t_m} =  (\tilde c^{\, t_*}_{i,t_m} )^TB_{t_*} c^{t_*}_{j,t_m}.  \]
Thanks to Lemma \ref{qq} and Lemma \ref{key} (iv), we know that $\tilde c^{\,t_*}_{j,t_m} = c^{\,t_*}_{j,t_m}=e_j$. Since $i$ does not appear in the mutation sequence, the $i$-th columns of $G^{\,t_*}_{t_m}$ and $\tilde G^{\,t_*}_{t_m}$ are $e_i$'s and thus the $i$-th rows of $\tilde C^{\,t_*}_{t_m}$ and $C^{\,t_*}_{t_m}$ are $e^T_i$'s which implies that the $i$-th coordinates of $\tilde c^{\,t_*}_{i,t_m}$ and $c^{\,t_*}_{i,t_m}$ are $1$'s and the other coordinates are non-negative. For simplicity, we write $\tilde c^{\,t_*}_{i,t_m}=(y_1,\dots,y_n)^T$ and $c^{\,t_*}_{i,t_m}=(x_1,\dots,x_n)^T$ with $x_i=y_i=1$ and $x_s,y_s\geq 0$ for $s\in [1,n]$.

Since $j$ is a source of $B_{t_*}$, then $b_{js,t_*}\geq 0$ and $b_{sj,t_*}\leq 0$ for $s\in [1,n]$. Therefore we have that
 \[    b_{ji,t_m} =  (\tilde c^{\,t_*}_{j,t_m})^T B_{t_*} c^{t_*}_{i,t_m}  =     e_j^T  B_{t_*}  (x_1,\dots,x_n)^T = b_{ji} + \sum\limits_{s\neq i}x_sb_{js,t^*}    \geq b_{ji} >0 ,      \]
and 
 \[    b_{ij,t_m} =  (\tilde c^{\,t_*}_{i,t_m})^T B_{t_*} c^{t_*}_{j,t_m}  =   (y_1,\dots,y_n)     B_{t_*}  e_j= b_{ij} + \sum\limits_{s\neq i}y_sb_{sj,t^*}    \leq b_{ij} <0.    \]
 In particular, we have that $-b_{ji,t_m}b_{ij,t_m}\geq -b_{ji}b_{ij}$.
\end{proof}

Let $B=(b_{ij})_{n\times n}$ be a totally sign-skew-symmetric matrix. For any subset $V\subset  [1,n]$, we denote $B^V$ the matrix obtained from $B$ by deleting the $i$-th row and $i$-th column of $B$ for every $i\notin V$, and we call $B^V$ the {\em submatrix of $B$ induced by $V$}.  
If $B$ is skew-symmetric, Muller gave the relations of scattering diagrams for $B^V$ and $B$ in  \cite{muller}. Of course, his proof works for skew-symmetrizable matrices.  For totally sign-skew-symmetric matrices, the relation of their $g$-fans is given in [Theorem 7.3, \cite{pan2023mutation}] which is a weaker version of Muller's result in [Theorem 33, \cite{muller}], and we refer to \cite{nakanishi2023pattern} for the details of $g$-fans.  

For a vector $v=(v_1,\dots,v_n)^T\in \mathbb R$, define \[ \supp(v):= \{i\in [1,n]\, \,|\,\, v_i\neq 0\}.\]  For any subset $V\subset  [1,n]$, denote $\pi_V: \mathbb R^n\rightarrow \mathbb R^V$ the natural projection.

Recall that in \cite{fomin2017introductionii}, a property of a totally sign-skew-symmetric matrix is called {\em hereditary} if it is preserved under restriction to any principal submatrix. The following lemma shows that the existence of maximal green sequences for totally sign-skew-symmetric matrices is a hereditary property.

\begin{lemma}[Hereditary property of existence of maximal green sequences]\label{submax}
	Let $B$ be a totally sign-skew-symmetric matrix and $B^V$ the submatrix of $B$ induced by $V\subset [1,n]$. Then a maximal green sequence (or a reddening sequence, respectively) of $B$ with the associated sequence of $c$-vectors $c_0,\dots,c_{m-1}$ induces a maximal green sequence (or a reddening sequence, respectively) of $B^V$ with the associated sequence of $c$-vectors $\pi_V(c_{i_1}),\dots,\pi_V(c_{i_k})$, where $c_{i_1},\dots,c_{i_k}$ is the sub-sequence of $c_0,\dots,c_{m-1}$ such that $\supp(c_{i_s})\subset V$ for $1\leq s\leq k$.  
	
	Furthermore, if a maximal green sequence $B$ begins with a sequence of mutations on vertices in $V$, then there is a maximal green sequence for $B^V$  which begins with the same sequence of mutations.
\end{lemma}
\begin{proof}
  It follows from Theorem 7.3 in \cite{pan2023mutation} and the same arguments of the proofs of Theorem 9 and Theorem 17 in \cite{muller}.
\end{proof}

\begin{theorem}[Target before Source Theorem]\label{tbsc}
		 Let $(B_t,C_t^{t_0},G_t^{t_0})_{t\in \mathbb T_n}$ be a matrix pattern at the initial vertex $t_0$ such that $B_{t_0} = (b_{ij})_{n\times n }$ is acyclic.  Suppose that $b_{ji}>0$ and $-b_{ji}b_{ij}\geq 4$. Let $(i_1,\dots,i_m)$ be a maximal green sequence of $B_{t_0}$, and let \[p = \min\{  s \,| \,  i_s = j, \,1\leq s\leq m     \},  \quad\text{and} \quad q = \min\{  s \,| \,  i_s = i, \,1\leq s\leq m     \}.\] Then $p>q$.
\end{theorem}
\begin{proof} We depict the maximal green sequence on $\mathbb T_n$ as follows.
		\[  
	(B_{t_0} ,C^{t_0}_{t_0},G^{t_0}_{t_0})
	\frac{i_1}{\quad\quad} (B_{t_1} ,C^{t_0}_{t_1},G^{t_0}_{t_1})\frac{i_2}{\quad\quad}\;  \dots\;  \frac{i_m}{\quad\quad}  (B_{t_m} ,C^{t_0}_{t_m},G^{t_0}_{t_m}) \]
	
    Assume that $p<q$.  Then $i_1, i_2,\dots,i_{p-1}\neq i,j$.  Let $B_{t_{p-1}}=(b'_{ij})_{n\times n}$. By Lemma \ref{xxx}, we have that $-b'_{ij}b'_{ji}\geq 4$ and $b'_{ji}\geq b_{ji}>0$. 
    By the Rotation Lemma, the sequence $(i_p, i_{p+1},\dots, \sigma^{-1}(i_{p-1}))$ is a maximal green sequence of $B_{t_{p-1}}$ beginning with $j$. Let $V=\{i,j\}$.  By Lemma  \ref{submax}, there is a maximal green sequence of $B_{t_{p-1}}^V$ beginning with $j$, which contradicts Lemma \ref{ranktwo}. Hence $p>q$.
\end{proof}

    The following two results were proved in \cite{brustle2017semipicture} for skew-symmetrizable matrices and here we give a proof for the totally sign-skew-symmetric matrix using Lemma \ref{submax}.
    Let $B$ be a totally sign-skew-symmetric matrix, and $(i_1,\dots,i_m)$ a reddening sequence of $B$ with the associated sequence of $c$-vectors $c_0,\dots,c_{m-1}$. We call the sub-sequence $(i_{k+1}, \dots, i_m)$ a {\em green tail of  $(i_1,\dots,i_m)$} if $c_{k},\dots,c_{m-1}$ are non-negative.

\begin{theorem}[Target before Source Theorem of $c$-vectors]\label{tbscc}
			 Let $(B_t, C_t^{t_0},G_t^{t_0})_{t\in \mathbb T_n}$ be a matrix pattern at the initial vertex $t_0$ such that $B_{t_0} = (b_{ij})_{n\times n }$.  Suppose that $b_{ji}>0$ and $-b_{ji}b_{ij}\geq 4$. Let $(i_1,\dots,i_m)$ be a reddening sequence of $B_{t_0}$ with the associated sequence of $c$-vectors $c_0,\dots, c_{m-1}$,  and let $(i_{k+1}, \dots, i_m)$ be a  green tail of  $(i_1,\dots,i_m)$.  If  both $e_j$ and $e_i$ appear in $c_k, \dots, c_{m-1}$, then $e_i$ appears before $e_j$.
\end{theorem}
\begin{proof} Assume that $e_j$ appears before $e_i$  in $c_k, \dots, c_{m-1}$.
	Let $V=\{i,j\}$ and $B^V$ the submatrix of $B$ induced by $V$.  By Lemma \ref{submax}, $B^V$ has a reddening sequence and this reddening sequence has a green tail beginning with $\pi_V(e_j)$ and the green tail must end with $\pi_V(e_i)$. This is impossible, since if  $\pi_V(e_j)$ and  $\pi_V(e_i)$ appear in a green tail of a reddening sequence of $B^V$, then $\pi_V(e_i)$ must appear before $\pi_V(e_j)$. For details, one may refer to Appendix 1 in \cite{brustle2014mgs}.
\end{proof}

The proof of the following result is the same as the proof of Corollary 3.3.2 in \cite{brustle2017semipicture}.

\begin{corollary}\label{not mutate at target of a multi-arrow}
  Consider any maximal green sequence on any totally sign-skew-symmetric matrix.  Then at each step, the mutation is at an index $j$ of the mutated matrix $B'$ such that there is no index $i$ such that  $-b'_{ij}b'_{ji}\geq 4$ and $b'_{ji}>0$.
\end{corollary}

\section*{Acknowledgment} 
The authors would like to thank the referee for careful reading and helpful comments. They also would like to thank Fang Li for his suggestions. Additionally, the second author wishes to acknowledge Ibrahim Assem and Shiping Liu for their help and support during his stay in Sherbrooke. This research is supported by the National Natural Science Foundation of China under Grant No. 12301048 and the Zhejiang Provincial Natural Science Foundation of China under Grant No. LQ24A010007.

\def\cprime{$'$} \def\cprime{$'$}
\providecommand{\bysame}{\leavevmode\hbox to3em{\hrulefill}\thinspace}
\providecommand{\MR}{\relax\ifhmode\unskip\space\fi MR }
\providecommand{\MRhref}[2]{%
	\href{http://www.ams.org/mathscinet-getitem?mr=#1}{#2}
}

\end{document}